\documentclass[10pt]{article}

\usepackage[USenglish]{babel} 
\usepackage[T1]{fontenc}
\usepackage[ansinew]{inputenc}

\usepackage{lmodern} 
\usepackage{amsthm,amsfonts,amssymb}
\usepackage{graphicx}
\usepackage{amsmath}

\usepackage{graphicx} 
\usepackage{subfig} 

\usepackage{amsmath}
\usepackage{amsthm}
\usepackage{amsfonts}
\usepackage{graphicx}
\usepackage{fancyhdr}

\newtheorem*{result*}{Main Result}
\newtheorem{mydef}{Definition}
\newtheorem{example}{Example}
\newtheorem{theorem}{Theorem}
\newtheorem{lemma}[mydef]{Lemma}
\newtheorem{corollary}[mydef]{Corollary}

\DeclareMathOperator{\New}{New}

\DeclareMathOperator{\Conv}{Conv}
\def\excise#1{}

\usepackage{setspace}
 
\begin{document}
\sffamily
\title{Bounding the Degree of Belyi Polynomials}
\author{Jose Rodriguez}
\date{November 12, 2011}
\maketitle

\begin{abstract}
Belyi's theorem states that a Riemann surface $X$, as an algebraic curve, is defined over $\overline{\mathbb{Q}}$ if and only if there exists a holomorphic function $B$ taking $X$ to $P^1\mathbb{C}$ with at most three critical values $\{0,1,\infty\}$.
By restricting to the case where $X=P^1\mathbb{C}$ and our holomorphic functions are  Belyi polynomials, we define a Belyi height of an algebraic number, $\mathcal{H}(\lambda)$, to be the minimal degree of  Belyi polynomials with $B(\lambda)\in\{0,1\}$.
Using the combinatorics of Newton polygons, we prove for non-zero $\lambda$ with non-zero  $p$-adic valuation, the Belyi height of $\lambda$  is greater than or equal to $p$
We also give examples of algebraic numbers which show our bounds are sharp.
\end{abstract}

\pagestyle{plain} 


\section{Introduction}

In this paper we fix an algebraic closure of $p$-adic numbers and denote it as $\overline{\mathbb{Q}}_p$.
We denote an embedded algebraic closure of the rational numbers in $\overline{\mathbb{Q}}_p$   as $\overline{\mathbb{Q}}$. 
A polynomial $B(x)$ in $\mathbb{\overline{Q}}[x]$ is said to have a critical point at $x_{i}$ if its derivative $B'(x)$ vanishes at $x_{i}$.
We say  $B(x)$ has a critical value of $B(x_i)$ when $x_i$ is a critical point.
A polynomial is said to be  a general Belyi polynomial if its critical values are contained in $\{0,1\}$.
Since composing a general Belyi polynomial  with any linear factor $(\gamma x-\alpha)$  yields another general Belyi polynomial, we normalize our set of polynomials by requiring   $B(0),\, B(1)\in\{0,1\}$. 
\begin{mydef}
A polynomial $B(x)\in\mathbb{\overline{Q}}[x]$ is said to be a 
normalized Belyi polynomial or Belyi polynomial if $B(0),\,B(1)\in\{0,1\}$ and $\{B(x_{i})\,:\, B'(x_{i})=0\}\subset\{0,1\}$.
\end{mydef} 
Equivalently we note that $B(x)$ is a Belyi polynomial if 
$$\begin{array}{l}
B(0),\, B(1)\in\{0,1\},\text{ and }\\
B'(x)\mid B(x)(1-B(x)).
\end{array}$$
We call these the two Belyi conditions.
With these conditions, a Belyi polynomial composed with a linear factor $(\gamma x-\alpha)$ is a Belyi polynomial if and only if $B(\gamma),B(\gamma-\alpha)\in\{0,1\}$.
For a fixed Belyi polynomial there exist finitely many linear factors we may compose with and yield a  Belyi polynomial.
This finiteness condition is essential to define our Belyi height with the property that there exist finitely many Belyi polynomials of a given degree.

\begin{example}\label{basic}
The simplest examples of  Belyi polynomials are $f(x)=x^n$, $f(x)=1-x$, and
$$B_{a,b}(x)=b^{b}a^{-a}(b-a)^{-(b-a)} x^{a}(1-x)^{b-a},\text{ where }a,b\in\mathbb{N}, { and } (b-a)\geq 0.$$
\end{example}
The Belyi polynomial $B_{a,b}(x)$ maps $\{\frac{a}{b},0,1\}$ to $\{0,1\}$. 
When we compose $B_{a,b}(x)$ with certain polynomials  $C(x)$ the result, $B_{a,b}(C(x))$, as fewer critical values than $C(x)$.
Specifically, when $C(x)$ satisfies the first Belyi condition and has a critical value of $\frac{a}{b}$, composing with $B_{a,b}$ reduces the number of critical values.

\begin{example} 
The Chebyshev polynomials of the first kind, $T_{n}(x),\, n\geq 1$,
$$\begin{array}{l}
   T_{0}(x)=1,\,\,\, T_{1}(x)=x,\,\,\, T_{n+1}(x)=2x\cdot T_{n}(x)-T_{n-1}(x)
\end{array}$$
have critical values contained in $\{-1,1\}$ and $T_n(1),T_n(-1)\in\{-1,1\}$. 
Therefore $\frac{1}{2}(T_{n}(x)+1)$ are general Belyi polynomials and $\frac{1}{2}(T_n(2x-1))+1)$ are Belyi polynomials.
\end{example}

This example is studied in detail in \cite{Bau06} where the normalization of Belyi polynomials is done with respect to $\{-1,1\}$ instead of $\{0,1\}$.

\begin{example}
 The composition of any two Belyi polynomials is a Belyi polynomial.
\end{example}

This example is a simple application of the chain rule and gives the set of Belyi polynomials a monoid structure under composition with identity, $x$.
This structure has been used to study the absolute Galois group in number theory \cite{wood06}, \cite{Ellen02} and  dynamical systems \cite{Pilgr00}.

Belyi polynomials belong to the larger set of Belyi functions.
A Belyi function $f$ maps a Riemann surface $X$ to the Riemann sphere $P^1\mathbb{C}$ with critical values contained in $\{0,1,\infty\}$.
Grothendieck was drawn into this subject because of Belyi's theorem \cite{Belyi83}, which states a Riemann surface $X$ is defined over $\overline{\mathbb{Q}}$ if and only if there exist a Belyi function mapping $X$ to $P^1\mathbb{C}$. 
This marked the beginning of his program on dessin d'enfants \cite{Schne94}, which is directly related to Belyi functions due to the well-known categorical equivalence between the two.

In the case where $X=P^1\mathbb{C}$ we normalize Belyi functions by requiring the set $\{0,1,\infty\}$ be mapped to $\{0,1,\infty\}$. 
As a corollary \cite{Bau06} of  the Riemann Existence Theorem \cite{Miran95} there exist finitely many normalized Belyi functions that map $P^1\mathbb{C}$ to $P^1\mathbb{C}$ of degree at most $n$, where degree is the cardinality of the pre-image of a point in 
$P^1\mathbb{C}\setminus\{0,1,\infty\}$.
This means there are finitely many normalized Belyi polynomials of a given degree, hence finitely many algebraic numbers mapped to zero or one by normalized Belyi polynomials of degree $d$.
The question we address in this paper is the following:  
for fixed $\lambda\in\overline{\mathbb{Q}}$, what is the minimal degree of normalized Belyi polynomials that map $\lambda$ to zero or one? 
We call this minimum the $\emph{Belyi\,height}$ of a number and denote it as $\mathcal{H}(\lambda)$. 
In \cite{Khadj02}, an upper bound of $\mathcal{H}(\lambda)$ is given, in addition to bounds for the case when  $X$  is an elliptic curve.
In this paper we will provide a sharp lower bound on the degree.
Our results follow directly from \cite{Zappo08} and \cite{Beckm89}.
As in Beckman's paper our result says bad reduction implies wild ramification. 
What this paper contributes is a proof which uses elementary combinatorial techniques and Newton polygons.
We will prove that Belyi polynomials with degree less than $p$ and $B(0)=0$  have Newton polygons with respect to $p$ (for the remainder of the paper all Newton polygons will be with respect to $p$)  contained in the Newton polygon of $B(x)-1$ [Theorem $\ref{containment}$].
We then prove the Newton polygon of $B(x)-1$ is contained in a single line segment [Theorem $\ref{line}$].
Using a classical lemma [Lemma $\ref{factorization}$] relating the Newton polygon of a polynomial to the $p$-adic valuation of its roots we prove:

\begin{result*}[Theorem \ref{result}]
The Belyi height of $\lambda$, $\mathcal{H}(\lambda)$, is greater than or equal to $p$ for $\lambda\neq 0$ in $\overline{\mathbb{Q}}$ with non-zero $p$-adic valuation.
\end{result*}

We remark that it is nontrivial to show that such a height is well defined, that is, for all algebraic numbers over $\mathbb{Q}$ there exists a Belyi polynomial, which maps it to either zero or one. 
Given $\lambda\in\overline{\mathbb{Q}}$ Belyi provided a way to construct  \cite{Schne94} a Belyi function, which maps $\{0,1,\lambda,\infty\}$ to $\{0,1,\infty\}$ by first constructing a polynomial  $g_\lambda(x)\in\mathbb{Q}[x]$ having rational critical values, $g_\lambda(\lambda)\in\mathbb{Q}$, and $\{0,1\}$  mapped to $\{0,1\}$. 
We compose $g_\lambda(x)$  with a linear factor $l_1(x)$, preserving the number of critical values, so that $l_1\circ g_\lambda (x)$ has a rational critical value $\frac{a_1}{b_1}$ between zero and one.
We compose $l_1(x)\circ g_\lambda(x)$ with   $B_{a_1,b_1}(x)$ so $B_{a_1,b_1}\circ l_1 \circ g_\lambda (x)$ has fewer critical values than $g_\lambda (x)$ as mentioned in Example \ref{basic}. 
Repeating this finitely many times yields a Belyi polynomial  $B_{a_k,b_k}\circ l_k \circ \cdots\circ B_{a_1,b_1}\circ l_1 \circ g_\lambda (x)$ that  maps $\lambda$ to a rational number $\frac{a_{k+1}}{b_{k+1}}$. 
We do a final iteration so that $\lambda$ is mapped to zero or one. While this algorithm  gives us a way of constructing Belyi polynomials it does not provide us a way of constructing all of them. 

\section{Newton Polygon Factorization}
We begin this section with an introduction to $p$-adic numbers, Newton polygons, and convex sets to state Lemma $\ref{factorization}$, which allows us to classify the roots of a polynomial using these objects (see \cite{Gouve97}, \cite{Weiss98}, \cite{Engle05} for a thorough introduction).
The $p$-adic metric on $\mathbb{Q}$ is defined as:
$$\left|\:\right|_{p}:
\begin{array}{ccc}
\mathbb{Q} & \rightarrow & \mathbb{R}\\
p^{k}\frac{a}{b} & \mapsto & p^{-k}
\end{array} $$
where $p\nmid ab\neq0$ and $\left|0\right|_{p}\equiv0$. 
The completion of $\mathbb{Q}$ under this metric will be denoted as $\mathbb{Q}_{p}$.
The algebraic closure of $\mathbb{Q}_p$ is denoted as $\overline{\mathbb{Q}}_p$  and has a $p$-adic absolute value.
Thus it makes sense to talk of the $p$-adic absolute value of any algebraic number over $\mathbb{Q}$.
Frequently, it will be easier to state results using the $p$-adic valuation 
$$\nu_{p}:\begin{array}{ccc}
\overline{\mathbb{Q}}_p & \rightarrow & \mathbb{R}\cup\infty\\
\lambda & \mapsto & -\log\left|\lambda\right|_{p}\end{array}$$
where $\nu_{p}(0)\equiv\infty$. 
The $p$-adic valuation has properties induced by the $p$-adic metric:

$$\begin{array}{l}
(1)\,\nu_{p}(\lambda_{1}\lambda_{2})=\nu_{p}(\lambda_{1})+\nu_{p}(\lambda_{2})\\
(2)\,\nu_{p}(\lambda)=\infty\, \text{ if and only if }\,\lambda=0\\
(3)\,\nu_{p}(\lambda_{1}+\lambda_{2})\geq\min\left\{ \nu_{p}(\lambda_{1}),\nu_{p}(\lambda_{2})\right\}.
\end{array}$$
The last property is induced because the $p$-adic metric is non-Archimedian meaning $$\left|a+b\right|_{p}\leq\max\{\left|a\right|_{p},\left|b\right|_{p}\}.$$

Define the valuation ring with respect to $p$ as $\mathcal{O}_{p}=\{\lambda\in\overline{\mathbb{Q}}_{p}:\,\nu_{P}(\lambda)\geq0\}$, the elements of the field $\overline{\mathbb{Q}}_{p}$ with non-negative valuation.
This ring has the maximal ideal $\mathfrak{m}_{p}=\{\lambda\in\overline{\mathbb{Q}}_{p}:\,\nu_{P}(\lambda)>0\}$, the elements of $\overline{\mathbb{Q}}_{p}$ with positive valuation. 
We denote the reduction map as $\pi:\mathcal{O}_{p}[x]\to\mathbb{F}[x] \text{ where } \mathbb{F}$ is the field $\mathcal{O}_{p}/\mathfrak{m_{p}}$.

The convex hull of a set of points is the intersection of all convex sets containing the points.
When we find the convex hull of finitely many points $\{(x_0,y_0),...,(x_n,y_n)\}\subset\mathbb{R}^2$, the result is a  point, line segment, or convex polygon described algebraically as
$$\{ (\sum_{i=0}^n c_i x_i,\sum_{i=0}^n c_i y_i)\in\mathbb{R}^2:\sum_{i=0}^n c_i=1\}$$
where $c_i\geq 0$ for all $i$. 
Given a polynomial $f(x)=a_{0}+a_{1}x+\cdots+a_{n}x^{n}$ over $\overline{\mathbb{Q}}$, then $\Conv_p(f)$ denotes the convex hull of $$\{\left(i,\nu_{p}(a_{i})\right)\in\mathbb{R}^{2}\,:\, a_{i}\neq0\}.$$
Our notation will be that  $[v_{i},v_{j}]$ denotes the line segment connecting the points $v_{i}$ and $v_{j}$. 
By convention $[v_{i},v_{i}]$ denotes the point $v_{i}$.
When $\Conv_p(f)$ is a polygon we label a subset of the polygon's vertices counter-clockwise from the left-most, $v_{0}$, ending at the right-most, $v_{m}$. 
The lower boundary of $\Conv_p(f)$ is the union of the $m$ line segments connecting $v_{i-1}$ to $v_{i}$ denoted as 
$$\bigcup_{i=1}^{m}[v_{i-1},v_{i}].$$ 
When $\Conv_p(f)$ is a line segment or point,  the lower boundary of $\Conv_p(f)$ is $[v_{0},v_{1}]$ or $[v_{0},v_{0}]$ respectively. 

\begin{mydef} 
The Newton polygon of a polynomial $f(x)\in\overline{\mathbb{Q}}[x]$ with respect to $p$, is the lower boundary of $\Conv_p(f)$:
 $$\New_{p}(f(x))=\begin{cases}
\begin{array}{l}
[v_{0},v_{0}]\\{}
[v_{0},v_{1}]\\
\bigcup_{i=1}^{m}[v_{i-1},v_{i}]\end{array} & \begin{array}{l}
\text{if }\Conv_{p}(f)\,\text{is\, the\, point\, }v_{0}\\
\text{if }\Conv_{p}(f)\,\text{is a line segment }\,[v_{0},v_{1}]\\
\text{if }\Conv_{p}(f)\text{ \text{is a polygon with lower boundary }}\bigcup_{i=1}^{m}[v_{i-1},v_{i}]\end{array}\end{cases}$$
\end{mydef}

The Newton polygon of a polynomial is a single vertex precisely when $f(x)$ is a monomial. When $f(0)=0$ the Newton polygons of $f(x)$ and $f(x)-1$ are closely related.

\begin{lemma}\label{newB-1}
Suppose $f(0)=0$ and $\New_p(f(x))=[v_{0},v_{1}]\cup...\cup[v_{m-1},v_{m}]$.

$(1)$ $\New_p(f(x)-1)=[v_{-1},v_{j}]\cup[v_{j},v_{j+1}]\cup...\cup[v_{m-1},v_{m}]$ for some $j$, $0\leq j \leq m$, and $v_{-1}$ denotes the origin.

$(2)$ Let $s_i$ denote the slope of $[v_{i-1},v_i]$ and $s_0$ denote the slope of $[v_{-1},v_j]$. Then $$s_1<...<s_j < s_0\leq s_{j+1}<...<s_m.$$

$(3)$ If the degree of $f(x)$ is less than $p$, then $\New_p(f(x))=\New_p(x\cdot f'(x))$.

\end{lemma}
\begin{proof}

This follows directly from properties of convex sets and the definition of Newton polygon. 
\end{proof}

When $f(0)=0$ and $v_j$ denotes the left-most point of $\New_p(f(x))\cap\New_p(f(x)-1)$, combinatorially, the first two parts of the lemma say:
the points in $\New_p(f(x))\cup\New_p(f(x)-1)$ to right of $v_j$ are in $\New_p(f(x))\cap\New_p(f(x)-1)$; 
the $\New_p(f(x)-1)$ has only one segment, $[v_{-1},v_j]$, to the left of $v_j$; 
and  the slope of $[v_{-1},v_j]$ is bounded by the slopes of line segments of $\New_p(f(x))$.
The third part combinatorial means that the $New_p(f'(x)$ is $New_p(f(x))$   but shifted to left one unit.

We will prove Theorem \ref{containment} and Theorem \ref{line} by taking full  advantage of the following classical lemma:

\begin{lemma}\label{factorization}
 Let $f(x)$ be a polynomial over $\overline{\mathbb{Q}}$ such that $\New_{p}(f(x))=[v_{0},v_{1}]\cup...\cup[v_{m-1},v_{m}]$. 
Let $s_{i}$ equal the slope of $[v_{i-1},v_{i}]$, and  $d_{i}$ equal the length of the projection of $[v_{i-1},v_{i}]$ to the $x$-axis.
Then the polynomial $f(x)$ may be written as
$$f(x)=a_{n}x^{d_{0}}f_{1}(x)\cdots f_{m}(x)$$ 
where $f_{i}(x)$ is monic with $d_{i}$ roots of valuation 
$-s_{i}$ counting multiplicity.
\end{lemma}

\begin{proof} 
 We refer to \cite{Weiss98}, p.74.
\end{proof}

We call this factorization of $f(x)$ its $\emph{Newton polygon factorization}$ with respect to $p$.


\begin{example}
 Setting $p$  equal to five, the Newton polygons of three polynomials, $h_1(x),h_2(x),h_3(x)$, are shown in bold.  The thin line segment is the left-most line segment of $\New_p(h_i (x)-1)$.

The left is an example where $\New(f(x))\not\subset\New(f(x)-1)$.  The center is an example where $\New(f(x)-1)$  is not contained in a line segment. The right is an example of a Newton polygon of a Belyi polynomial.

\center{\includegraphics[scale=0.8]{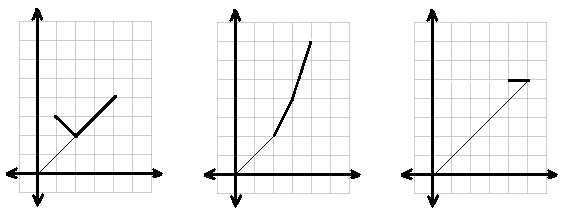}}


$$\begin{array}{lll}h_{1}=5^{4}\cdot x\left(x-\frac{1}{5}\right)^{2}\left(x-5\right), & h_{2}=5^{7}x^{2}\left(x-\frac{1}{5^{2}}\right)\left(x-\frac{1}{5^{3}}\right), & h_{3}=\frac{5^{5}}{4^{4}}\cdot x^{4}\left(1-x\right),
\\h_{1}=5^{4}x^{4}-3^{3}\cdot5^{3}x^{3}+51\cdot5^{2}x^{2}-5^{3}x, & h_{2}=5^{7}x^{4}-6\cdot5^{4}\cdot x^{3}+5^{2}x^2, & h_{3}=-\frac{5^{5}}{4^{4}}x^{5}+\frac{5^{5}}{4^{4}}x^{4}\end{array}$$

\end{example}

\section{Newton Polygons of Belyi Polynomials}
We prove in the case where $B(x)$ is a  Belyi polynomial of degree less than $p$  with zero as a root that $\New_p(B(x))\subset\New_{p}(B(x)-1)$. We then prove $\New_p(B(x)-1)$  is contained in a single line segment. 
Using these two results we are able to give a lower bound on the Belyi height. 
\begin{theorem}\label{containment}
  If $B(x)\in\overline{\mathbb{Q}}[x]$ is a Belyi polynomial of degree less than $p$ such that $B(0)=0$, then $\New_{p}(B(x))\subset \New_{p}(B(x)-1)$.
\end{theorem}

\begin{proof} 
Let $B(x)=\sum a_k x^k$. 
If $B(x)$ is a monomial the result is trivial so we consider the case where $\New_{p}(B(x))=[v_{0},v_{1}]\cup...\cup[v_{m-1},v_{m}]$, $m>0$. 
Using the same notation as Lemma  $\ref{newB-1}$, we may assume $\New_{p}(B(x)-1)=[v_{-1},v_{j}]\cup[v_{j},v_{j+1}]\cup...\cup[v_{m-1},v_{m}]$ and 
$$s_1<...<s_j < s_0\leq s_{j+1}<...<s_m.$$
Our goal is to show $v_{j}=v_{0}$ and the result follows. Lemma  $\ref{factorization}$ allows us to write 
$$\begin{array}{l}
B(x)=a_{n}x^{d_{0}}f_{1}(x)\cdots f_{m}(x),\\
B(x)-1=a_{n}g_{0}(x)g_{j+1}(x)\cdots g_{m}(x)
\end{array}$$
where every root of the monic polynomials $f_i$  and $g_i$, $i\neq0$, has valuation $-s_i$.
In addition $\deg(g_{i})=\deg(f_{i})$ when $i>j$, while each root of $g_{0}$ has valuation $-s_{0}$.

Since $\deg(B(x))<p$ and $B(0)=0$ then $\nu_{p}(a_{i})=\nu_{p}(i\cdot a_{i})$ for every $a_{i}\neq0$.
Therefore $\Conv_p(B(x))=\Conv_p (xB'(x))$ and $\New_{p}(B(x))=\New_{p}(xB'(x))$.
Hence $B'(x)$ may be written as 
$$B'(x)=\deg(B)a_{n}x^{d_{0}-1}h_{1}(x)\cdots h_{m}(x)$$
where $\deg(h_{i})=\deg(f_{i})$ and $h_{i}$ is monic with roots of valuation $-s_{i}$. 
By the Belyi conditions, $$\deg(B)a_{n}x^{d_{0}-1}h_{1}(x)\cdots h_{m}(x)\mid a_{n}x^{d_{0}}f_{1}(x)\cdots f_{m}(x)\cdot a_{n}g_{0}(x)g_{j+1}(x)\cdots g_{m}(x).$$

Therefore $h_{i}(x)\mid f_{i}(x)$ when $i\leq j$. 
Since the degrees of the monic polynomials are also equal it follows $f_{1}=h_{1}$ when $j\geq 1$. Taking the derivative of $B(x)$ and substituting $f_1(x)$ for $h_1(x)$ we have  
$$\deg(B)a_{n}x^{d_{0}-1}f_{1}(x)h_{2}(x)\cdots h_{m}(x)=f_{1}'(x)(a_{n}x^{d_{0}}f_{2}(x)\cdots f_{m}(x))+f_{1}(x)(a_{n}x^{d_{0}}f_{2}(x)\cdots f_{m}(x))'$$
and so
$$f_{1}(x)\mid f_{1}'(x)(a_{n}x^{d_{0}}f_{2}(x)\cdots f_{m}(x)).$$ Because $f_1(x)$ and $f_i(x)$ share no common roots when $i\neq 1$, $f_{1}(x)\mid f_{1}'(x)$, yielding a  contradiction when $j\geq 1$. Hence $v_0=v_j$.
\end{proof}

Next we show that if $B(x)$ is a  Belyi polynomial such that $B(0)=0$, then $\New(B(x)-1)$ must be a line segment, and in preparation prove two lemmas. 

\begin{lemma}\label{derivaivedivides0} 
 Suppose $f(x)$ is a nonzero polynomial over an algebraically closed field of characteristic zero. If $f'(x)$ divides $f(x)^{2}$ and $f(0)=0$ then $f(x)=a_{n}x^{d}$. 
\end{lemma}

\begin{proof}
 Suppose $f(x)=a_{n}\prod_{i=1}^{m}(x-\alpha_{i})^{d_{i}}$ where $\alpha_{i}$ are distinct. 
Then 
$$f'(x)=a_{n}\prod_{i=1}^{m}(x-\alpha_{i})^{d_{i}-1}g(x),\text{ where } g(x)=\sum_{i=1}^{m}d_{i}(x-\alpha_1)(x-\alpha_2)\cdots \widehat{(x-\alpha_i)}\cdots (x-\alpha_m)$$ 
and $\widehat{(x-\alpha_i)}$ denotes omitting a term.
Note that  $\deg(g(x))=m-1$ and the coefficient of the leading term is $\deg(f(x))$. 
For each root $\alpha_{j}$ of $f(x)$, $g(\alpha_{j})\neq 0$. 
But $g(x)$ also divides $f(x)^{2}$ so $g(x)$ must have degree zero and $m=1$. 
\end{proof}
The same proof holds in characteristic $p$ if every $d_{i}$ is not divisible by $p$ and $p\nmid\deg(f(x))$, giving us:
\begin{corollary}\label{derivativedividesall}
Suppose $f(x)$ is a nonzero polynomial over an algebraically closed field of arbitrary characteristic. 
If $f'(x)$ divides $f(x)^{2}$ and $f(0)=0$ then $f(x)=a_nx^{d}$ or $\deg(f(x))\geq p$. 
\end{corollary}

\begin{lemma}\label{upper}
Given $f(x)\in\overline{\mathbb{Q}}$ of degree $n$, $f(0)=0$, 
$\New_{p}(f(x))=[v_{0},v_{1}]\cup...\cup [v_{m-1},v_{m}]$, and $m>0$, then there exists $\gamma$ such that $R(x)=\frac{1}{a_{n}\gamma^{n}}f(\gamma x)$ with $\New_{p}(R(x))=[w_{0},w_{1}]\cup...\cup[w_{m-1},w_{m}]$ has $[w_{m-1},w_{m}]$ contained in the $x$-axis. 
\end{lemma}

 \begin{proof}  
 The polynomial $f(x)$ has Newton factorization $f=a_{n}f_{1}(x)...f_{m}(x)$ with the roots of $f_{m}(x)$ of least valuation.
Let $\gamma$ be a root of $f_{m}(x)$. 
For $R(x):=\frac{1}{a_{n}\gamma^{n}}f(\gamma x)$ the roots are of the form $\frac{\gamma_{i}}{\gamma}$ where $\gamma_{i}$ is a root of $f_{i}(x)$. 
Therefore the valuation of a root of $R(x)$ equals $\nu_{p}(\gamma_{i})-\nu_{p}(\gamma)\geq0$.
It follows the slopes of $[w_{i-1},w_{i}]$ of $\New_{p}(R(x))$ are less than zero if $i\neq m$ and equal to zero when $i=m$. 
Since $R(x)$ is monic and $[w_{m-1},w_{m}]$ has slope zero then  $w_{m}$ and $w_{m-1}$ are in the $x$-axis. 
 \end{proof}

\begin{theorem}\label{line}
 If $B(x)$ is a Belyi polynomial of degree less than $p$ with $B(0)=0$, then $New_{p}(B-1)$ is a line segment. 
\end{theorem}

\begin{proof}
If $B(x)$ is a monomial the result is trivial. 
Now suppose $B(x)=a_{1}x+a_{2}x^{2}+...+a_{n}x^{n}$ with Newton factorization $a_{n}x^{d_{0}}f_{1}(x)...f_{m}(x)$. 
By Lemma \ref{newB-1}, using the already defined notation from Theorem $\ref{containment}$, we see  for $m\geq 1$ 
$$\begin{array}{c}
 \New_{p}(B(x))=[v_{0},v_{1}]\cup...\cup[v_{m-1},v_{m}],\\
 \New(B(x)-1)=[v_{-1},v_{j}]\cup[v_{j},v_{j+1}]\cup...\cup[v_{m-1},v_{m}].
\end{array}$$
By Theorem $\ref{containment}$, $j=0$ so the slopes satisfy $$s_0\leq s_1<s_2<\cdots <s_m.$$
To prove the theorem we must show that $m=1$ and $s_0=s_1$.
Let $\gamma$ be a root of $B(x)$ with least valuation. 
This is a root of $f_{m}(x)$ and $\nu_p(\gamma)=-s_m$.
Let $R(x):=\frac{1}{a_{n}\gamma^{n}}B(\gamma x)$. 
Then the Newton factorization of $R(x)$ is 
$$R(x)=\frac{(\gamma x)^{d_0}}{\gamma^{d_{0}}}\frac{f_{1}(\gamma x)}{\gamma^{d_{1}}}...\frac{f_{m}(\gamma x)}{\gamma^{d_{m}}}$$

By Lemma $\ref{upper}$, 
$$\New_p(R(x))=[w_{0},w_{1}]\cup...\cup[w_{m-1},w_{m}]$$
has $[w_{m-1},w_{m}]$ contained in the $x$-axis.
So $R(x)$ has $d_m$ roots of valuation zero. 
Since the slope of each $[w_{i-1},w_{i}]$, is non-positive
it follows $\New_{p}\left(R(x)\right)$ is contained in the upper half plane, hence $R(x)$ is in $\mathcal{O}_{p}[x]$ as is each of its factors.

As in Theorem $\ref{containment}$, $B(x)-1$ has a Newton factorization 
$$B(x)-1=a_{n}g_{0}(x)g_{1}(x)\cdots g_{m}(x)$$
where $\deg g_i =d_i$ and $g_i$ has roots of valuation $-s_i$.
So
$$R(x)-\frac{1}{a_{n}\gamma^{n}}=\frac{1}{a_{n}\gamma^{n}}\left(B(\gamma x)-1\right)=\frac{g_{0}(\gamma x)}{\gamma^{d_{0}}}\frac{g_{1}(\gamma x)}{\gamma^{d_{1}}}...\frac{g_{m}(\gamma x)}{\gamma^{d_{m}}}\in\mathcal{O}_{p}[x]$$
and  $R(x)-\frac{1}{a_{n}\gamma^{n}}$ also has $d_i$ roots of valuation $(s_i-s_m)$.

Since $R(x)-\frac{1}{a_{n}\gamma^{n}}$  is monic, the product of its roots is $\frac{(-1)^{n+1}}{a_n \gamma^n}\in\mathcal{O}_{p}$, and
$$0\leq \nu_p(\frac{(-1)^{n+1}}{a_n \gamma^n})=\nu_p(\frac{-1}{a_n \gamma^n})=-d_0(s_0-s_m)-d_1(s_1-s_m)-\cdots-d_m(s_m-s_m).$$
With this, we see $\nu_p(\frac{-1}{a_n\gamma^n})=0$ if and only if $(s_i-s_m)=0$.
So in the case where $\nu_p(\frac{(-1)^n}{a_n\gamma^n})=0$ it follows $m$ is necessarily one and $s_0=s_1$.

We conclude the proof by using the reduction map and Corollary $\ref{derivativedividesall}$ to show that the remaining case where
$\nu_p(\frac{-1}{a_n\gamma^n})>0$ leads to a contradiction. 
Since $\deg R(x)<p$ and $R(0)=0$ then $\New_p(R(x))=\New_p(x\cdot R'(x))$. So $R'(x)$ also has leading coefficient and $d_m$ roots of valuation zero. 
In particular $R(x)$, $R(x)-\frac{1}{a_n\gamma^n}$, $R'(x)$, and each of their factors are in $\mathcal{O}_p[x]$ and 
$\pi(R(x))$, $\pi(R(x)-\frac{1}{a_n\gamma^n})$, $\pi(R'(x))$ are nonzero.  
The Belyi condition $B(x)\mid B(x)(B(x)-1)$ imply $R'(x)\mid R(x)(R(x)-\frac{1}{a_n\gamma^n})$.
So $\pi(R'(x))\mid \pi(R(x))\pi(R(x)-\frac{1}{a_n\gamma^n})$. 
But when $\nu_p(\frac{-1}{a_n\gamma^n})>0$,
$\pi(R(x))=\pi(R(x)-\frac{1}{a_n\gamma^n})$. 
We can then apply Corollary $\ref{derivativedividesall}$, which says $\pi(R(x))$ has no nonzero roots. 
But this contradicts the fact that $R(x)$ has $d_m$ roots of valuation zero.

\end{proof}


\begin{theorem}\label{result}
The Belyi height of $\lambda$, $\mathcal{H}(\lambda)$, is greater than or equal to $p$ for $\lambda\neq 0$ in $\overline{\mathbb{Q}}$ with non-zero $p$-adic valuation.
\end{theorem}

\begin{proof}
If $B(0)=1$, then consider the Belyi polynomial $1-B(x)$, so without loss of generality we may assume $B(0)=0$.
If $\deg(B(x))<p$, then by Theorem $\ref{line}$, $\New(B(x))$ and $\New(B(x)-1)$ are contained in a single line segment. 
Therefore all non-zero roots of $B(x)$ and $B(x)-1$ have the same valuation. 
This means $\nu_p(1)=\nu_p(\lambda)=0$, a contradiction when $\nu_p(\lambda)\neq 0$.
Hence, $\deg B(x)$ must be greater than or equal to $p$.
\end{proof}

With this theorem, we know for every Belyi polynomial with rational number  $\frac{a}{b}$ in lowest terms as a root will have degree greater than or equal to every prime $p$ that divides $ab$.
The well-known Belyi polynomial from Example $\ref{basic}$
$B_{a,b}(x)$
has as its critical points $\{\frac{a}{b},0,1\}$.
Therefore 
$B_{1,p}(\frac{x}{p})$ 
is a normalized Belyi polynomial mapping $p$ to zero showing our bound is sharp.
However, in general, it is not true  $\mathcal{H}(a)\geq a$, for $a \in\mathbb{Z}$, as the following example shows.

\begin{example}\label{counter}

If we consider the Belyi polynomial $B(x)=-\frac{1}{4}(x-1)^2(x-4)$, then $B(4)=0$ and $\mathcal{H}(4)\leq3$. By Theorem $\ref{result}$, $\mathcal{H}(4)\geq 2$.
A direct calculation by solving a quadratic shows that  $\mathcal{H}(4)\neq 2$, so it follows $\mathcal{H}(4)=3$.
\end{example}

We end with a few open questions. 
First, how can one express $\mathcal{H}:\overline{\mathbb{Q}}\rightarrow\mathbb{R}^+$  in a closed form?
By Example $\ref{counter}$ we know this  is not a simple function such as $\mathcal{H}(a)=a$ when we restrict $\mathcal{H}$ to the natural numbers.
Second, when  is $\mathcal{H}(ab)\geq \max\{\mathcal{H}(a),\mathcal{H}(b)\}$ for $a,b\in\mathbb{Z}^+$?
By  Theorem $\ref{result}$ we know $\mathcal{H}(pq)\geq \max\{\mathcal{H}(p),\mathcal{H}(q)\}$ for primes $p$ and $q$.
Third, we ask for fixed $h\in\mathbb{R}^+$ how many distinct $\lambda$ satisfy the inequality 
$\mathcal{H}(\lambda)\leq h\in\mathbb{R}^+$?
In addition, can we adjust the definition of Belyi height so that the number of such $\lambda$ grows on the order of a polynomial as we vary $h$. 
Finally, does there exist a unique Belyi polynomial of degree equal to $\mathcal{H}(\lambda)$ with $\lambda$ as one of its roots? If not can we classify such polynomials, and do they have the same Newton polygon?

I would like to especially thank  Eric Katz for supervising this research and the Ronald E. McNair Postbaccalaureate Achievement Program for funding this project.

\bibliography{belyi}{}
\bibliographystyle{plain}
\end{document}